\documentclass{ECACRM}

\usepackage{mathtools}
\usepackage{mathrsfs}

\usepackage{caption}
\usepackage{subcaption}
\usepackage{ECApack} 

\usepackage{enumitem}

\renewcommand{\.}{\hskip 0.7pt}

\newcommand{\C}{\mathcal{C}}
\newcommand{\D}{\mathcal{D}}
\newcommand{\Y}{{\scriptstyle\mathcal{Y}}}
\newcommand{\Su}{\mathcal{S}}
\newcommand{\Lu}{\mathscr{L}\!}

\DeclareMathOperator{\Conv}{conv}
\DeclareMathOperator{\Ex}{Ex}
\DeclareMathOperator{\Exterior}{ext}
\DeclareMathOperator{\RR}{R}
\DeclareMathOperator{\GG}{G}
\DeclareMathOperator{\R}{r}
\DeclareMathOperator{\G}{g}
\DeclareMathOperator{\Sle}{s}
\DeclareMathOperator{\Tot}{t}
\DeclareMathOperator{\W}{w}
\DeclareMathOperator{\A}{A}
\DeclareMathOperator{\B}{B}
\DeclareMathOperator{\M}{\mathcal{M}}

\DeclareMathOperator{\aand}{\;\wedge\;}

\newcommand{\rex}{\RR_{\Exterior}}

\DeclarePairedDelimiter\floor{\lfloor}{\rfloor}
\DeclarePairedDelimiter\ceil{\lceil}{\rceil}

\newtheorem{theorem}{Theorem}
\newtheorem{lemma}[theorem]{Lemma}

\newtheorem{prop}[theorem]{Proposition}

\theoremstyle{definition}
\newtheorem{definition}[theorem]{Definition}

\newtheorem{observation}[theorem]{Observation}

\newtheorem{conj}[theorem]{Conjecture}

\begin{document}

\papertitle[Convex Grabbing Game]{Massively Winning Configurations in the Convex~Grabbing Game on the Plane}

\paperauthor{Martin Dvorak}
\paperaddress{Department of Theoretical Computer Science and Mathematical Logic, \newline
Faculty of Mathematics and Physics, Charles University; \newline 
and Institute of Science and Technology, Austria;}
\paperemail{martin.dvorak@matfyz.cz}

\paperauthor{Sara Nicholson}
\paperaddress{Department of Applied Mathematics, \newline
Faculty of Mathematics and Physics, Charles University;}
\paperemail{s.nicholson@mail.utoronto.ca}

\makepapertitle

\Summary{The convex grabbing game is a game where two players, Alice and Bob, alternate taking extremal points from the convex hull of a point set on the~plane. Rational weights are given to the points. The goal of each player is to maximize the total weight over all points that they obtain. We restrict the setting to the case of binary weights. We show a construction of an arbitrarily large odd-sized point set that allows Bob to obtain almost $3/4$ of the total weight. This construction answers a question asked by Matsumoto, Nakamigawa, and Sakuma in [Graphs and Combinatorics, 36/1\,(2020)]. We also present an arbitrarily large even-sized point set where Bob can obtain the entirety of the total weight. Finally, we discuss conjectures about optimum moves in the convex grabbing game for both players in general.}

\section{Introduction}

The graph grabbing game, first presented by Winkler~\cite{graph_grabbing}, is a game where two players alternate removing non-cut vertices from a vertex-weighted graph. This game has been studied \cite{graph_sharing,grab2, gold_grabbing} and led to variants including the convex grabbing game by Matsumoto, Nakamigawa, and Sakuma~\cite{matsumoto_convex_2020}, which we discuss here.

A cake $\C$ is determined by a set of points, which we call \textit{cherries}, that lie in general position in the Euclidean plane. Each cherry $c$ is given a weight $\W(c) \in \mathbb{Q}$. There are two players in this game: \textit{Alice} and \textit{Bob}. They alternate selecting cherries from the set of remaining cherries $C \subseteq \C$, with Alice going first. They can only select extremal points of the convex hull of $C$, defined by $\Ex(C) = \{ c \in C ~|~ c \notin \Conv(C \setminus \{c\}) \}$, and the selected cherry is removed from the set $C$. The game is over when all cherries are taken.

If $|\C|$ is even, we say that $\C$ is an \textit{even-sized cake}. Similarly, if $|\C|$ is odd, we say that $\C$ is an \textit{odd-sized cake}.
\bigskip

\noindent We denote the sequence of moves by Alice as:
\[(a_1,a_2, \dots, a_{\ceil*{\frac{|\C|}{2}}})\]
We denote the sequence of moves by Bob as:
\[(b_1,b_2, \dots, b_{\floor*{\frac{|\C|}{2}}})\] 
This results in a gameplay
\[ \mathbf{q} = (a_1, b_1, a_2, b_2, \dots, a_{\frac{|\C|}{2}}, b_{\frac{|\C|}{2}}) \]
on an even-sized cake $\C$, or
\[ \mathbf{q} = (a_1, b_1, a_2, b_2, \dots, a_{\floor*{\frac{|\C|}{2}}}, b_{\floor*{\frac{|\C|}{2}}}, a_{\ceil*{\frac{|\C|}{2}}}) \]
on an odd-sized cake $\C$. In~the~end, each player obtains a total score equal to the sum of the weights of the cherries they selected. In particular, we define the total gain of Alice as:
\[ \A(\mathbf{q}) =\!\! \sum\limits_{i \.\in\. \{1, 2, \dots, \ceil*{\frac{|\C|}{2}}\} } \!\!\!\W(a_i) \]
The objective of Alice is to maximize $\A(\mathbf{q})$. The objective of Bob is to minimize $\A(\mathbf{q})$.
We also define the complement 
\[ \B(\mathbf{q}) =\!\! \sum\limits_{i \.\in\. \{1, 2, \dots, \floor*{\frac{|\C|}{2}}\} } \!\!\!\W(b_i) \]
We observe that $ \A(\mathbf{q}) + \B(\mathbf{q}) $ is invariant of $\mathbf{q}$; it is constant for a given cake $\C$.
Alice wants to minimize $\B(\mathbf{q})$ and, naturally, Bob wants to maximize $\B(\mathbf{q})$. We will work with $\B(\mathbf{q})$ a lot because we will focus on maximizing Bob's results --- which looks like a harder task, at least at first glance.

Finally, we define the \textit{minimax result} on the cake $\C$, denoted by $\M(\C)$, as the total gain of Bob if both players play optimally throughout the whole game.
\[ \M(\C) = \!\min_{a_1 \in\. \Ex(\C)} 
\left( \max_{\substack{b_1 \in\. \Ex( \\ \C\. \setminus \{a_1\})}}
\left( \min_{\substack{a_2 \in\. \Ex( \\ \C\. \setminus \{a_1,b_1\})}} 
\left( \max_{\substack{b_2 \in\. \Ex( \\ \C\. \setminus \{a_1,b_1,a_2\})}}
\left(  \!\dots\!  \biggl( \!
\B\!\. \rule{0mm}{9mm} ((a_1, b_1, a_2, b_2, \dots)) \! \biggr) \!\dots\! \right)
\! \right) \! \right) \! \right) \]

We focus on a restricted version of this game where only $\{0,1\}$ weights are considered. Any cherry $c \in \C$ where $\W(c) = 1$ is called \textit{red}; and we define $\RR(\C) \coloneqq \{ c_r \in \C ~|\. \W(c_r) = 1\}$. In a similar manner, any cherry $c \in \C$ where $\W(c) = 0$ is called \textit{green}; and we define $\GG(\C) \coloneqq \{c_g \in \C ~|\. \W(c_g) = 0\}$. We thereby have $\RR(\C) \cup \GG(\C) = \C$ and $\RR(\C) \cap \GG(\C) = \emptyset$.

Furthermore, for any $C \subseteq \C$, we define values $\R(C) \coloneqq |\RR(C)\.|$ and $\G(C) \coloneqq |\GG(C)\.|$. Note that $\R(C) + \G(C) = |C|$ and that $\R(C) = \sum_{c \in C} \W(c)$.

Matsumoto, Nakamigawa, and Sakuma \cite{matsumoto_convex_2020} posed the question of finding the \hbox{maximum} possible value for $\M(\C) - (\R(\C) - \M(\C))$ on an odd-sized cake, that is, how much can Bob win by? In Section~\ref{sun}, we present, for any natural number $z$, a construction of an odd-sized cake $\C$ such that $\R(\C) = 4z+2$; and we provide a tactic for Bob which guarantees $\M(\C) \geq \frac{3}{4}\R(\C) - \frac{1}{2}$. Therefore, Bob can win by an arbitrarily large margin.

In Section~\ref{moon}, we show that there exists an even-sized cake~$\C$ where $\R(\C) = m$ and $\M(\C) = m$ for every $m \in \mathbb{N}$; that is, Bob can obtain all red cherries.

\pagebreak
\section{Order types}

In this section, we provide a combinatorial point of view on the convex grabbing game. The following definition has been adapted from \cite{ordertypes}.
\smallskip

\begin{definition}
    Given a tuple $(p, q, r)$ of three distinct cherries, we define their \textit{orientation} $\nabla pqr$ as $+1$ if the sequence $(p, q, r)$ traverses the triplet $\{p, q, r\}$ in a counterclockwise direction, and as $-1$ if this direction is clockwise.
    
    Consider two cakes, $P$ and $Q$, where $|P| = |Q|$. We say that a bijection $\pi : P \rightarrow Q$ is \textit{order-preserving} if $\W(c) = \W(\pi(c))$ for each cherry $c \in P$ and there exists a sign $\sigma \in \{-1, +1\}$ such that $\nabla \pi(p)\pi(q)\pi(r) = \sigma \cdot \nabla pqr$ for all sequences $(p, q, r)$ of three distinct cherries in $P$.
    
    If such a bijection exists, we say that $P$ and $Q$ are \textit{order-equivalent}. We see that this relation is reflexive (using the identity; $\sigma_{P,P} = 1$), symmetric (using the inverse bijection; $\sigma_{Q,P} = \sigma_{P,Q}$), and transitive (using the compositions of bijections; $\sigma_{O,Q} = \sigma_{P,Q} \cdot \sigma_{O,P}$). The resulting equivalence classes are called \textit{order types}.
\end{definition}

\noindent We will show that the result of the convex grabbing game depends on the order~type~only, and so the game could be defined using order types instead of cakes.
\bigskip

\begin{lemma} \label{insidetriangle}
    Let $x, y, z, p \in \mathbb{R}^2$ be four distinct points in the plane. The point $p$ lies inside the triangle $\Conv(\{x,y,z\})$ if and only if $\nabla xyz = \nabla xyp = \nabla xpz = \nabla pyz$.
\end{lemma}

\begin{proof}
    The condition $\nabla xyz = \nabla xyp$ is equivalent to saying that the point $p$ lies in the open half-plane determined by the line $\overline{xy}$ and the point $z$. Similarly, $\nabla xyz = \nabla xpz$ if and only if $p$ lies in the open half-plane determined by the line $\overline{xz}$ and the point $y$; and $\nabla xyz = \nabla pyz$ if and only if $p$ lies in the open half-plane determined by the line $\overline{yz}$ and the point $x$. The intersection of these three open half-planes is precisely the inner area of the triangle $\Conv(\{x,y,z\})$.
\end{proof}
\smallskip

\begin{lemma} \label{extremalcommute}
    Let $\C$ and $\D$ be two order-equivalent cakes whose order-preserving bijection is $\pi : \C \rightarrow \D$. Let $C \subseteq \C$ and $D = \pi(C) \subseteq \D$. Then $\Ex(D) = \pi(\Ex(C))$.
\end{lemma}

\begin{proof}
    Let $c \in C$. Consider the following five statements.

    $$V_1(c) \equiv\quad \Bigl(~ c \in \Ex(C) ~\Bigr)$$
    
    \begin{multline*}
        V_2(c) \equiv\quad \neg \Bigl(~ \exists (x,y,z) \in (C \setminus \{c\})^3 :\;\; x \neq y \neq z \neq x \aand
        \\ \aand  \nabla xyz = \nabla xyc = \nabla xcz = \nabla cyz ~\Bigr)
    \end{multline*}
    \begin{multline*}
        \!\!\!V_3(c) \equiv~ \neg \Bigl(~ \exists (x,y,z) \in (C \setminus \{c\})^3 :\;\; x \neq y \neq z \neq x \aand
        \\ \aand \nabla \pi(x)\pi(y)\pi(z) = \nabla \pi(x)\pi(y)\pi(c) = \nabla \pi(x)\pi(c)\pi(z) = \nabla \pi(c)\pi(y)\pi(z) ~\Bigr)
    \end{multline*}
    
    \pagebreak
    \begin{multline*}
        V_4(c) \equiv\quad \neg \Bigl(~ \exists (x',y',z') \in (D \setminus \{\pi(c)\})^3 :\; x' \neq y' \neq z' \neq x' \aand
        \\ \aand \nabla x'y'z' = \nabla x'y'c\. = \nabla x'c\.z' = \nabla c\.y'z' ~\Bigr)
    \end{multline*}
    
    $$V_5(c) \equiv\quad \Bigl(~\pi(c) \in \Ex(D) ~\Bigr)$$

    \bigskip

    $V_1(c)$ is equivalent to $V_2(c)$ by combining Lemma~\ref{insidetriangle} with Carathéodory's theorem.
    We know that $V_2(c)$ is equivalent to $V_3(c)$ because $\pi$ is order-preserving.
    We know that $V_3(c)$ is equivalent to $V_4(c)$ because $\pi$ is bijective.
    $V_4(c)$ is equivalent to $V_5(c)$ again by Lemma~\ref{insidetriangle} and Carathéodory's theorem.
    
    After applying the same argument for all cherries in $c \in C$, we obtain the desired identity $\Ex(D) = \pi(\Ex(C))$.
\end{proof}

\begin{prop}
    If $\C$ and $\D$ are order-equivalent cakes, then $\M(\C) = \M(\D)$. 
\end{prop}

\begin{proof}
    We use induction on the cake size $|\C| = |\D|$.
    
    If $|\C| = |\D| < 4$, then $\Ex(\C)=\C$ and $\Ex(\D)=\D$. It is optimal to take the cherries in the order of decreasing weights for both players. We easily see $\M(\C) = \M(\D)$. 
    
    If $|\C| = |\D| \ge 4$, we use the induction hypothesis for $|C| = |D| = |\C| - 2 = |\D| - 2$. Let $\pi : \C \rightarrow \D$ be the order-preserving bijection. The calculation goes as follows. 

    \[ \M(\D) = \!\min_{a_1 \in\. \Ex(\D)} 
    \left( \max_{\substack{b_1 \in\. \Ex( \\ \D\. \setminus \{a_1\})}}
    \left( \min_{\substack{a_2 \in\. \Ex( \\ \D\. \setminus \{a_1,b_1\})}} \!
    \left( \max_{\substack{b_2 \in\. \Ex( \\ \D\. \setminus \{a_1,b_1,a_2\})}} \!
    \left(  \!\dots\!  \biggl( \!
    \B\!\. \rule{0mm}{9mm} ((a_1, b_1, a_2, b_2, \dots)) \! \biggr) \!\dots\! \right)
    \! \right) \! \right) \! \right) \]
    
    \[  \stackrel{\text{from~def.}}{=} ~~~ \min_{a_1 \in\. \Ex(\D)} 
    \left( \max_{\substack{b_1 \in\. \Ex( \\ \D\. \setminus \{a_1\})}}
    \Bigl( \W(b_1) + \M(\D\. \setminus \{a_1,b_1\})  \Bigr) \! \right)  \]
    
    \[  \stackrel{\text{Ind.~hyp.}}{=} ~~~ \min_{a_1 \in\. \Ex(\D)} 
    \left( \max_{\substack{b_1 \in\. \Ex( \\ \D\. \setminus \{a_1\})}}
    \Bigl( \W(b_1) + \M(\C\. \setminus \{\pi^{\!-1}(a_1), \pi^{\!-1}(b_1)\})  \Bigr) \! \right)  \]
    
    \[  \stackrel{\text{Lemma~\ref{extremalcommute}}}{=} ~~~  \min_{a_1 \in\. \pi(\Ex(\C))} 
    \left( \max_{\substack{b_1 \in\. \pi(\Ex( \\ \C\. \setminus \{\pi^{\!-1}(a_1)))}}
    \Bigl( \W(\pi^{\!-1}(b_1)) + \M(\C\. \setminus \{\pi^{\!-1}(a_1), \pi^{\!-1}(b_1)\})  \Bigr) \! \right)  \]
    
    \[  \stackrel{\substack{a'_1 \coloneqq \pi^{\!-1}(a_1) \\ b'_1 \coloneqq \pi^{\!-1}(b_1)}}{=} ~~~ \min_{a'_1 \in\. \Ex(\C)} 
    \left( \max_{\substack{b'_1 \in\. \Ex( \\ \C\. \setminus \{a'_1\})}}
    \Bigl( \W(b'_1) + \M(\C\. \setminus \{a'_1,b'_1\})  \Bigr) \! \right) 
    ~~~ \stackrel{\text{from~def.}}{=} ~~~ \M(\C) \]
    
    \noindent We have thereby verified $\M(\C) = \M(\D)$.
\end{proof}

\pagebreak
\section{Sun configuration} \label{sun}

We present a family of odd-sized cakes which we call the sun configuration, and we show that, from any cake $\C$ in the family, Bob will obtain at least $\frac{3}{4}\R(\C) - \frac{1}{2}$ red cherries given that he follows a certain tactic.
\medskip

\begin{definition}
    We define a \textit{beam} $\Y$ as four cherries in the order [\.green, red, green, red\.] lying on an arc (see Figure \ref{fig:ray}).
\end{definition}

\begin{figure}[h]
    \centering
    \begin{subfigure}{.22\textwidth}
        \centering
        \includegraphics[width=.8\linewidth]{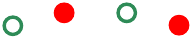}
        \caption{Beam ($\Y$).}
        \label{fig:ray}
    \end{subfigure}%
    \begin{subfigure}{.68\textwidth}
        \centering
        \includegraphics[width=.8\linewidth]{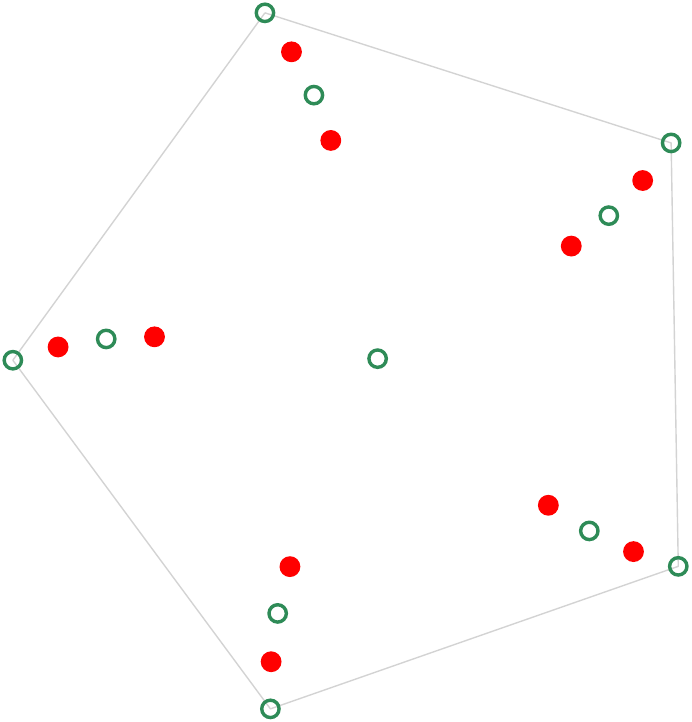}
        \caption{Sun ($\mathcal{S}_{\.5}$).}
        \label{fig:sun}
    \end{subfigure}
    \caption{Sun configuration.}
\end{figure}

\begin{definition}
    Let $k > 2$ be an odd integer. We define the \textit{sun} as a cake $\Su_{k}$ with $k$ beams and an additional green cherry $\zeta$ in the centre (see Figure~\ref{fig:sun} for an example $\Su_{5}$) such that:
    \begin{itemize}[leftmargin=10mm]
        \item The sun is rotationally symmetric with the $k$ beams evenly spaced around the centre~$\zeta$.
        \item Each beam is far enough from the centre such that, with the removal of any proper subset $Y$ of the cherries on the beam, the outermost cherry on the beam will always be in $\Ex(\Su_{k} \setminus Y)$.
        \item Consider any beam $\Y_i$ (see Figure~\ref{fig:sun_cut}). A line drawn through any two cherries of $\Y_i$ does not cut through any other beam and it keeps $\frac{k-1}{2}$ beams from $\Su_{k} \setminus \Y_i$ on each side. Additionally, if we consider $\Y_i \cup \{\zeta\}$, then they are all in convex position.
    \end{itemize}
    We have constructed a sun $\Su_k$ where $\R(\Su_k) = 2k$ and $\G(\Su_k) = 2k+1$.
\end{definition}

\begin{figure}
  \centering
  \includegraphics[width=12cm]{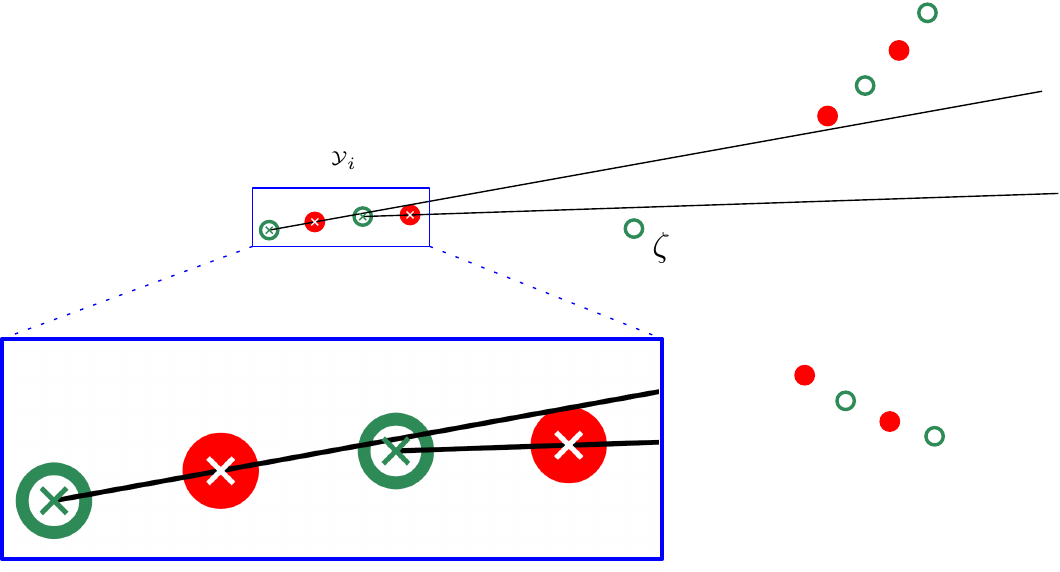}
  \caption{Highlighting details for the sun.}
  \label{fig:sun_cut}
\end{figure}

\pagebreak
\begin{definition} \label{carefulgreedy}
    In the convex grabbing game on a sun, we say that a player follows the \textit{Careful greedy tactic} if the player chooses a move according to these instructions:
    \medskip
    \texttt{
        ~\\Is there an extremal red cherry? \\
            \indent YES $\longrightarrow$ Is there an extremal red cherry on a beam that also contains\\ \indent a red cherry that is not extremal? \\
                \indent \indent YES $\longrightarrow$ Take the extremal red cherry from this beam. \\
                \indent \indent NO $\longrightarrow$ Is there a beam with a single extremal red cherry? \\
                \indent \indent \indent YES $\longrightarrow$ Take this extremal red cherry. \\
                \indent \indent \indent NO $\longrightarrow$ Take any extremal red cherry. \\
            \indent NO $\longrightarrow$ Is there any beam with at least one green cherry and \\
            \indent no red cherries on it? \\
            \indent \indent YES $\longrightarrow$ Take an extremal green cherry from this beam. (at least \\ 
            \indent \indent one cherry from each remaining beam is extremal in any moment) \\
            \indent \indent NO $\longrightarrow$ FAIL!
    }
\end{definition}
\smallskip

\begin{theorem} \label{maintheorem}
    From the sun $\Su_k$, Bob will get at least $\frac{3k-1}{2} = \frac{3}{4} \R(\Su_k) - \frac{1}{2}$ red cherries by using the Careful greedy tactic, no matter how Alice plays. As a result, we obtain the desired property $\M(\Su_k) \ge \frac{3}{4} \R(\Su_k) - \frac{1}{2}$.
\end{theorem}

We approach the proof as follows. We let Bob follow the Careful greedy tactic in all his moves. Alice can do anything.

We always describe the game state by the set of remaining cherries $C \subseteq \Su_k$ and create a lower bound for how many red cherries Bob is guaranteed to obtain from this moment until all cherries are taken. We characterize the set $C$ by the existence of a certain line (see Definition~\ref{defihalfplanes}) and by quantities $\R(C)$, $\Sle(C)$, and $\Tot(C)$ (see Definition \ref{defisingles}).

We start our proof with Lemma~\ref{halfplaysecond} in order to calculate what will happen in the second phase of the game. Lemma~\ref{halfplayfirst} then analyzes the first phase of the game while using the result of Lemma~\ref{halfplaysecond} in order to obtain the sum of Bob's score over both phases. In the end, we prove Theorem~\ref{maintheorem} as a straightforward corollary of Lemma~\ref{halfplayfirst}.

\pagebreak
\begin{definition} \label{defihalfplanes}
    In a moment of a gameplay on the sun $\Su_{k}$, denote $C \subseteq \Su_{k}$ as the set of remaining cherries. For all lines that pass through $\zeta$, we define the set $\mathcal{U}_{\.C}$ as the set of all closed half-planes defined by these lines.
    If $C \subseteq U \in \mathcal{U}_{\.C}$, then $U$ is called a \textit{bounding half-plane} for $C$.
\end{definition}

\begin{lemma} \label{invariantfirstphase}
    If Bob has been following the Careful greedy tactic from the beginning of the game on the sun $\Su_k$, then, for $C \subseteq \Su_k$ in a moment of gameplay when it is Alice's turn and a bounding half-plane does not yet exist, we have:
        \begin{enumerate}
            \item $\zeta \notin \Ex(C)$
            \item Each beam $\Y$ is either fully remaining (that is $\Y \subseteq C$), or fully removed (that is $\Y \cap C = \emptyset$), or exactly the two innermost cherries (one green, one red) remain.
        \end{enumerate}
\end{lemma}

\begin{proof} 
    Item (1): From the hyperplane separation theorem; if $\zeta$ were an extremal point of $C$, there would be a bounding half-plane going directly through $\zeta$.
    
    Item (2): We proceed by induction on the number of taken cherries.
    Base case: $\Su_k$ satisfies the properties since each beam is fully remaining.
    Assume that this holds up till some set $C_0 \subseteq \Su_k$, and it is Alice's turn.
    Induction step: From $C_0$ Alice can only take a green cherry from some beam $\Y$. The beam $\Y$ is either fully remaining, or has exactly the two innermost cherries remaining by the induction hypothesis. After Alice takes the green cherry, a red cherry will be revealed on beam $y$. This will be the only red cherry available, and so by following the Careful greedy tactic, Bob will take this cherry. Therefore beam $\Y$ will end up either fully removed, or with the two innermost cherries remaining. These two moves will give $C_1 \subset C_0$ where $|C_1| = |C_0| - 2$ which either will maintain all properties or a bounding half-plane will have emerged.  
\end{proof}

\begin{lemma} \label{lemmanofail}
    If Bob always follows the Careful greedy tactic, he will never reach FAIL. 
\end{lemma}

\begin{proof}
    Before a bounding half-plane emerges, this is clear from Lemma~\ref{invariantfirstphase}.
    After a bounding half-plane emerges, leaving $C \subset \Su_{k}$, there will always be a beam $\Y$ such that all remaining cherries of $\Y$ are in $\Ex(C)$; therefore, if there are no extremal red cherries, the beam $\Y$ will have at least one extremal green cherry and no red cherries --- Bob can take a green cherry here. Therefore, Bob will never reach the FAIL branch. 
\end{proof}

\begin{definition}
    A beam $\Y$ is \textit{semi-exposed} in $C$ if $|\Y \cap \RR(C)| = 2$ and, at the same time, $|\Y \cap \RR(C) \cap \Ex(C)| = 1$.
\end{definition}

\begin{lemma} \label{invariant}
    If Bob has been following the Careful greedy tactic from the beginning of the game on the sun $\Su_k$, then, when it is Alice's turn, there will never be a semi-exposed beam.
\end{lemma}

\begin{proof} 
    Assume to the contrary that a semi-exposed beam exists after a move by Bob and that it is the first time this happens. Lemma~\ref{invariantfirstphase} shows it cannot happen before a bounding half plane has emerged. Suppose now we are at a point in the gameplay when a bounding half-plane exists and there is a semi-exposed beam after Bob's move.
    
    In any moment of a gameplay on the sun $\Su_{k}$, removing a cherry can either (1) reveal a beam in full, or (2) reveal no new cherries, or (3) reveal a single cherry. Clearly, neither (1) nor (2) can produce a semi-exposed beam. 
    
    In (3), a semi-exposed beam can be produced either by taking the centre cherry $\zeta$, or by taking a green cherry from a beam which contains two red cherries (this beam then becomes semi-exposed). Neither of these moves can be performed by Bob because the Careful greedy tactic allows taking a green cherry only if it is from a beam and this beam does not contain any red cherries. We are left only with the option that it was Alice who generated a semi-exposed beam by (3).
    
    Since taking a red cherry from a semi-exposed beam is the top priority in Bob's Careful greedy tactic, the only possible reason for Bob leaving a semi-exposed beam is if Alice leaves two semi-exposed beams after her turn. Alice can produce a \hbox{semi-exposed} beam by (3); however, (3) can only produce a single semi-exposed beam at a time. Therefore, there was at least one semi-exposed beam before Alice's move, which is a contradiction with this being the first time there is a semi-exposed beam after Bob's move.
\end{proof}

\begin{definition}  \label{defisingles}
    In a moment of a gameplay on the sun $\Su_{k}$, denote $C \subseteq \Su_{k}$ as the set of remaining cherries. We define $\Sle(C)$ as the number of beams in $C$ which have a \textit{single} remaining red cherry. Furthermore, we define $\Tot(C)$ as the number of beams in $C$ which have \textit{at least one} remaining red cherry.
\end{definition}

\begin{lemma} \label{halfplaysecond}
    Let $C \subseteq \Su_k$ be a remaining subset of the sun obtained by Bob following the Careful greedy tactic such that $|C|$ is odd and a bounding half-plane for $C$ exists. It is now Alice's turn. 
       
    From the set $C$, Bob will obtain at least $\frac{1}{2}(\R(C) - \Sle(C))$ red cherries from now until the end of the game (using the Careful greedy tactic).
\end{lemma}

\begin{proof}
    This can be proved by induction on $|C|$.
    
    If $|C| \le 1$, then $\R(C) = \Sle(C)$, thus the statement holds trivially (it says that Bob will get at least $0$ red cherries). 
    
    If $|C| > 1$, then we assume the statement holds for $|C'| = |C| - 2$. We proceed by case analysis.

    \begin{itemize}[leftmargin=10mm]
        \item If Alice first takes a green cherry, then the lemma holds no matter what Bob does.
        \begin{itemize}[leftmargin=6mm]
            \item If Bob proceeds by taking a green cherry as well, then we see $\R(C') = \R(C)$ and $\Sle(C') = \Sle(C)$, thus it reduces exactly to the induction hypothesis.
            \item If Bob proceeds by taking a red cherry, then we observe $\R(C') = \R(C) - 1$ and $|\Sle(C') - \Sle(C)| \le 1$, thus $\frac{1}{2}(\R(C') - \Sle(C')) \ge \frac{1}{2}(\R(C) - \Sle(C)) - 1$. By the induction hypothesis, Bob will obtain at least $\frac{1}{2}(\R(C) - \Sle(C)) - 1$ red cherries in the future. Since Bob has just taken one red cherry, Bob obtains at least $\frac{1}{2}(\R(C) - \Sle(C))$ red cherries in total.
        \end{itemize}
        \item If Alice first takes a red cherry, then we need to consider the properties of the game state and Bob's strategy in order to show that the lemma holds.  
        \begin{itemize}[leftmargin=6mm]
            \item If Alice took the red cherry from a beam with only this red cherry, then we see $\R(C') = \R(C) - 1$ and $\Sle(C') = \Sle(C) - 1$, thus $\frac{1}{2}(\R(C) - \Sle(C)) = \frac{1}{2}(\R(C') - \Sle(C'))$. The rest follows by applying the induction hypothesis in the same way as in the previous cases.
            \item If, prior to Alice's move, there were two red cherries on the beam, then Bob's Careful greedy tactic leads to taking the other red cherry from the same beam. By Lemma~\ref{invariant}, the second red cherry is guaranteed to be extremal. This gives $\R(C') = \R(C) - 2$ and $\Sle(C') = \Sle(C)$, thus the induction hypothesis guarantees that Bob will be able to obtain at least $\frac{1}{2}(\R(C') - \Sle(C')) = \frac{1}{2}(\R(C) - \Sle(C)) - 1$ future red cherries; and, since Bob has just taken one red cherry, Bob obtains at least $\frac{1}{2}(\R(C) - \Sle(C))$ red cherries in total. \qedhere
        \end{itemize} 
    \end{itemize} 
\end{proof}

\begin{lemma} \label{halfplayfirst}
    Let $C \subseteq \Su_k$ be a remaining subset of the sun obtained by Bob following the Careful greedy tactic such that $|C|$ is odd and a bounding half-plane for $C$ does not yet exist. It is now Alice's turn. 
        
    There exists a half-plane $U \!\in \mathcal{U}_{\.C}$ such that Bob will obtain at least $\R(C) - \Tot(U \cap C)$ red cherries from now until the end of the game by using the Careful greedy tactic.
\end{lemma}

\begin{proof}
    We proceed by induction on $\R(C)$. Note that all extremal cherries are green, by Lemma~\ref{invariantfirstphase}, and each of them lies on a beam that has a red cherry (in particular, the beam has the same number of green and red cherries).
    
    Bob's Careful greedy tactic dictates to always take the neighboring red cherry from the same beam as Alice just took her cherry from. Finally, as our base case, we utilize Lemma~\ref{halfplaysecond} once a bounding half-plane emerges.
    
    Induction step: Consider $U$ from the induction hypothesis. Let $C'$ be the set of cherries remaining from $C$ after Alice's move and Bob's move, $|C'| = |C| - 2$. The induction hypothesis provides that Bob will get at least $\R(C') - \Tot(U \cap C')$ cherries during the remainder of the game. Since $\R(C) = \R(C') + 1$ and $\Tot(U \cap C) \ge \Tot(U \cap C')$, the difference between the lemma statement and the number from the induction hypothesis~is at most one. However, Bob has just taken a red cherry, so the lemma statement is satisfied.
    
    Base case: Once a bounding half-plane emerges, Bob is guaranteed to obtain at least $\frac{1}{2}(\R(C) - \Sle(C))$ more red cherries by Lemma~\ref{halfplaysecond}. We set $U$ to be this bounding half-plane, thus $C = U \cap C$. We observe $\R(C) = 2 \cdot \Tot(C) - \Sle(C)$. In other words, $\Sle(C) = 2 \cdot \Tot(U \cap C) - \R(C)$. Using this equality, Lemma~\ref{halfplaysecond} can be rewritten as: Bob is guaranteed to obtain at least $\frac{1}{2}(\R(C) - (2 \cdot \Tot(U \cap C) - \R(C)))$ more red cherries. That is equal to $\R(C) - \Tot(U \cap C)$. This is what we wanted to prove.
\end{proof}

\noindent \textbf{Theorem \ref{maintheorem}}~(restated). \textit{
    From the sun $\Su_k$, Bob will get at least $\frac{3k-1}{2} = \frac{3}{4} \R(\Su_k) - \frac{1}{2}$ red cherries by using the Careful greedy tactic, no matter how Alice plays. As a result, we obtain the desired property $\M(\Su_k) \ge \frac{3}{4} \R(\Su_k) - \frac{1}{2}$.}

\begin{proof}
    Given the properties of $\Su_k$, we see that any half-plane \hbox{$U\!\.\! \in \mathcal{U}_{\.\Su_k}$} has 
    \[\Tot(U \cap \Su_k) \le \frac{k+1}{2} \]
    and thus, by Lemma~\ref{halfplayfirst}, Bob is guaranteed to obtain at least
    \[ \R(\Su_k) - \Tot(U \cap \Su_k) \,\ge\, 2k - \frac{k+1}{2} \,=\, \frac{3k-1}{2} \]
    red cherries using his Careful greedy tactic. This provides the lower bound for the minimax result.
\end{proof}
\bigskip

\section{Moon configuration} \label{moon}

We present a family of even-sized cakes which we call the moon configuration, on which Bob can easily obtain all red cherries.

\begin{definition}
	Let $n \in \mathbb{N} \setminus \{0, 1\}$. We define the \textit{moon} $\Lu_n$ (see Figure~\ref{fig:moon} for an example $\Lu_6$) as follows.
	
	Choose a centre point $S$ and draw two circles; $\alpha(S, 1)$, called outer; and $\beta(S, 1\!-\!\varepsilon)$, called inner, where $0 < \varepsilon < 1 - \cos \left( 90^\circ / n \right)$. Draw $n$ lines through $S$ such that they are rotationally symmetric with a period of $180^\circ/n$. Pick one line, called the main line. The main line defines two half-planes. The ``upper'' half-plane is discarded. The ``other'' half-plane will create the moon.
	
	Place a green cherry at each intersection of any line with the outer circle $\alpha$. Place a red cherry at each intersection of any line except the main line with the inner circle~$\beta$.
	We have constructed a cake $\Lu_n$ such that $\R(\Lu_n) = n-1$ and $\G(\Lu_n) = n+1$, where $\Ex(\Lu_n) = \GG(\Lu_n)$.
\end{definition}

\begin{figure}[h]
	\centering
	\includegraphics[width=12cm]{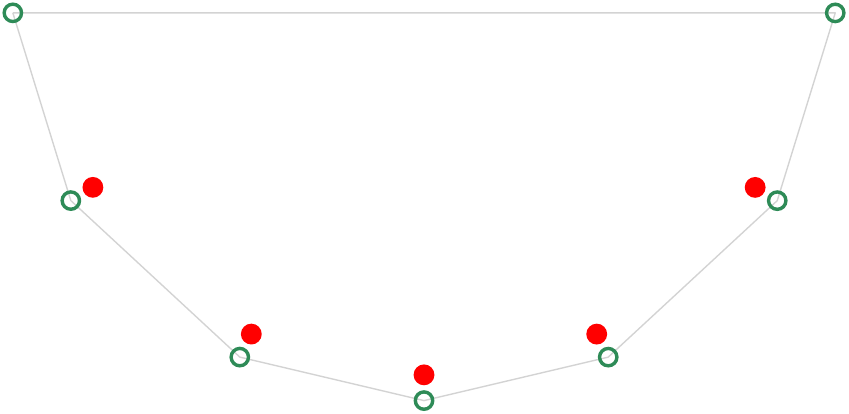}
	\caption{Moon ($\Lu_6$).}
	\label{fig:moon}
\end{figure}

\begin{observation} \label{moon-observations}
	We see that $\Lu_{n}$ has the following properties: the removal of a green cherry will reveal a (single) red cherry,  and the set $C \subset \Lu_{n}$, $|C| = |\Lu_n|-2$, obtained by the removal of a green cherry followed by a removal of a red cherry will be order-equivalent to $\Lu_{n-1}$.
\end{observation}

\begin{definition} \label{simplegreedy}
    We say that a player follows the \textit{Simple greedy tactic} if the player chooses a move according to this rule:
    \medskip
    \texttt{
        ~\\Is there an extremal red cherry? \\
            \indent YES $\longrightarrow$ Take an extremal red cherry. \\
            \indent NO $\longrightarrow$ Take any extremal cherry.
    } 
\end{definition}

\begin{theorem} \label{moontheorem}
    From the moon $\Lu_n$, if Bob follows the Simple greedy tactic, he will obtain all red cherries. This results in $\M(\Lu_n) = \R(\Lu_n) = n-1$.
\end{theorem}

\begin{proof}
	We proceed by induction on $n$.
	
	In the case with $n=2$, the moon will have four cherries in total; three extremal green cherries and one red cherry lying inside their triangle; therefore, Alice can select any green cherry and Bob will take the only red cherry by the Simple greedy tactic.
	
	Assume that $n > 2$ and the theorem holds for $\Lu_{n-1}$. For Alice's first move $a_1$, there are only green cherries available, and so she will take one of them. By Observation~\ref{moon-observations}, this will reveal a single red cherry, hence Bob, by following the Simple greedy tactic, will always take this red cherry for his first move $b_1$. 
	
	By Observation~\ref{moon-observations}, the remaining set of cherries $\Lu_{n} \setminus \{a_1,b_1\}$ is order-equivalent to $\Lu_{n-1}$.
	Therefore, by the induction hypothesis, Bob will obtain all $n-2$ red cherries from $\Lu_{n} \setminus \{a_1,b_1\}$; and, since he already took a red cherry in his first move, from $\Lu_{n}$ he obtains a total of $n-2+1 = n-1$ red cherries. 
\end{proof}

\section{Miscellaneous}

In order to obtain configurations that are favourable for Alice on even-sized and odd-sized cakes, a single red cherry can be placed outside the convex hull for the sun configuration and moon configuration respectively. Adding the extra red cherry swaps the parity of our constructions. We obtain the following cakes $\C$ and $\D$ that are good for Alice.
    
\begin{observation}
    Theorem~\ref{maintheorem} implies that there exists an even-sized cake $\C$ such that $\M(\C) \le \frac{1}{4}\R(\C) + \frac{1}{4}$. And, in a similar manner, Theorem~\ref{moontheorem} implies that there exists an odd-sized cake $\D$ with any desired $\R(\D) \in \mathbb{N}$ such that $\M(\D) = 0$.
\end{observation}

\bigskip

Furthermore, we would like to know what the optimal gameplay looks like in general. We came up with the following conjectures regarding the tactics which each player could employ in order to select their next move. 

\begin{conj} \textit{Greedy-move conjecture}. If $\Ex(C) \cap \RR(C) \neq \emptyset$, there exists a move that takes a red cherry from $\Ex(C) \cap \RR(C)$ such that the move~is~optimal.
\end{conj}

\noindent Note that the Careful greedy tactic (see Definition~\ref{carefulgreedy}) and the Simple greedy tactic (see Definition~\ref{simplegreedy}) are refinements of what the Greedy-move conjecture says.

\begin{conj} \textit{Strong greedy-move conjecture}. If $\Ex(C) \cap \RR(C) \neq \emptyset$, then every move that takes a red cherry is optimal.
\end{conj}

\noindent We will soon show that, even though we don't know whether the Greedy-move conjecture and the Strong greedy-move conjecture hold, we can easily prove that the former implies the latter (while the other implication holds trivially).

\begin{conj}
    \textit{No-reveal-move conjecture}. If $\Ex(C) \cap \RR(C) = \emptyset$ and we have a set of non-revealing moves $N = \{c \in \Ex(C) ~|~ \Ex(C \!\setminus\! \{c\}) \cap \RR(C) = \emptyset \}$ that is not empty, then there exists $c \in N$ such that selecting $c$ is optimal.
\end{conj}

\noindent We later found a counterexample that disproved the No-reveal-move conjecture, which we will show soon.

\begin{prop} \label{greedystronggreedy}
	The Greedy-move conjecture implies the Strong greedy-move conjecture.
\end{prop}

\begin{proof}
	Consider the following set of red cherries $\rex(C) = \{c \in \RR(C) ~|~ c \notin \Conv(\GG(C)) \} $. 
	We prove that ``the Greedy-move conjecture implies the Strong greedy-move conjecture'' by induction on $|\rex(C)|$. If $|\rex(C)|=1$, both conjectures are trivially equivalent.
	
	Assume that the Greedy-move conjecture holds in general and that the Strong \hbox{greedy-move} conjecture holds for up to $|\rex(C)| = n-1$ red cherries. We want to prove that the Strong \hbox{greedy-move} conjecture holds for up to $|\rex(C)| = n$ red cherries. Seeking contradiction, assume that Alice has two possible moves taking a red cherry $c_i, c_j \in \rex(C) \cap \Ex(C)$ that lead to different outcomes $\B(\mathbf{q})$.
	
	If Alice starts by taking $c_i$, then by the induction hypothesis, $c_j$ is among Bob's optimal moves. If Alice starts by taking $c_j$, then by the induction hypothesis, $c_i$ is among Bob's optimal moves. Either way, this leaves $C' = C\setminus\{c_i, c_j\}$. In the first case, Alice ends up with $\W(c_i) + \R(C') - \M(C')$ red cherries. In the second case, Alice ends up with $\W(c_j) + \R(C') - \M(C')$ red cherries. Since they are both equal to $1 + \R(C') - \M(C')$, we obtain a contradiction.
\end{proof}

\pagebreak
\begin{prop} \label{norevealfalse}
	The No-reveal-move conjecture is false.
\end{prop}

\begin{proof}
	We show a sketch of the proof through the construction in Figure~\ref{fig:cex}. 
	
	In this construction, the only non-revealing first move is to select $a_1 = c_2$.
	If, in the gameplay $\mathbf{q}$, Alice starts by taking this green cherry $c_2$, giving $\mathbf{q} = (c_2, b_1, \dots, b_5)$, then Bob can select $c_4$, giving $\mathbf{q} = (c_2, c_4, a_2,\dots,b_5)$, and they end up with $A(\mathbf{q})=1$ because $\M(C \setminus \{c_2, c_4\}) = 3$.
	
	However, if Alice selects $a_1 = c_1$ for her first move, she reveals two red cherries at the same time. Alice is therefore able to take a red cherry in her second move, after Bob moves. For Bob's first two moves, in order for Alice to not obtain a second red cherry on her third move, he has to take one of the two red cherries which Alice revealed in her first move, and $c_2$; however, the order of Bob selecting these does not matter. If Alice selects $a_3 = c_3$, she once again reveals two red cherries, and she is then guaranteed to be able to select a second red cherry in her fourth move. Therefore, by not selecting the non-revealing cherry in her first move, Alice is able to get a result of $A(\mathbf{q}) = 2$.
\end{proof}

\begin{figure}[h]
	\centering
	\includegraphics[width=13cm]{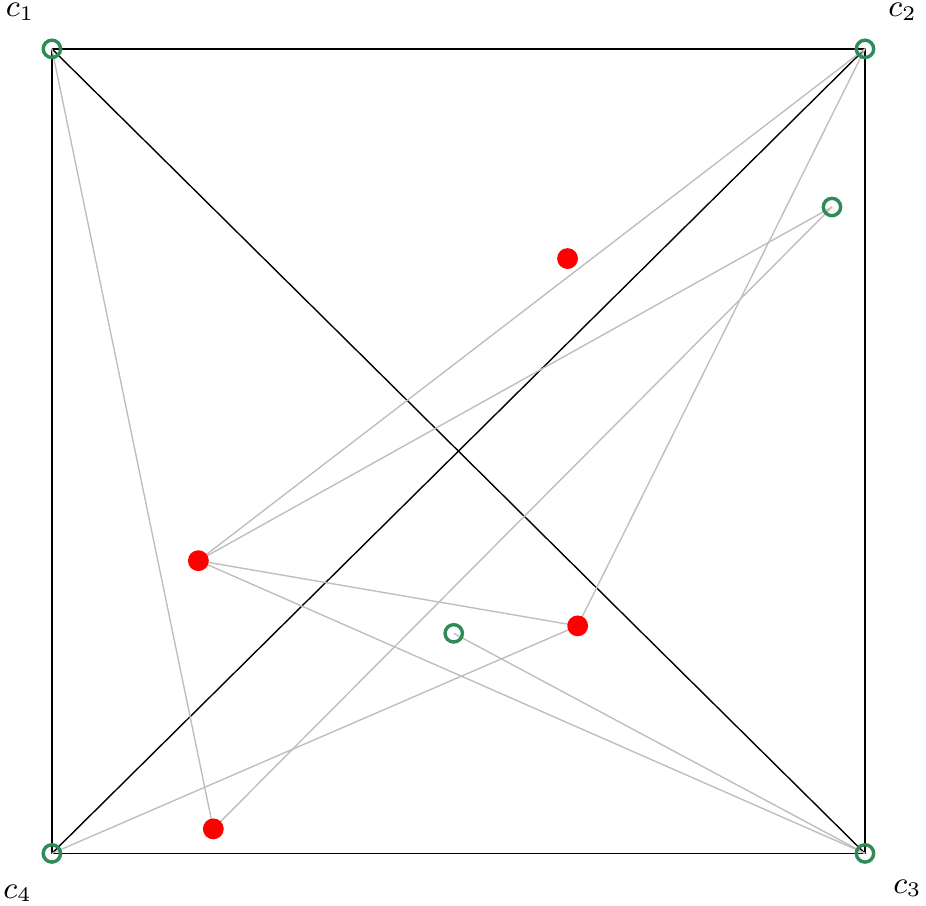}
	\caption{Counterexample to the No-reveal-move conjecture, with lines added for visual aid.}
	\label{fig:cex}
\end{figure}

\section{Conclusion}

We solved the open problem from \cite{matsumoto_convex_2020} by providing a construction that builds an \hbox{odd-sized} cake $\Su_k$ with $\M(\Su_k) - (\R(\Su_k) - \M(\Su_k)) \ge x$ for any $x \in \mathbb{N}$.
Now consider the value:
\[ \gamma = \limsup_{p \,\rightarrow\, \infty} \left( \max_{\substack{\text{odd-sized} \\ \text{cake } \C}} \left\{ \frac{\M(\C)}{\R(\C)} ~ \middle| ~ \R(\C) = p \right\} \right) \]
Our construction provides a lower bound $\gamma \ge \frac{3}{4}$. On the other hand, \cite{matsumoto_convex_2020} shows that Alice can always obtain at least one red cherry on any odd-sized cake. However, this only gives the trivial upper bound $\gamma \le 1$. We pose a new open question of determining the value~$\gamma$.
\bigskip

Analysis of gameplays would be easier if the state space of need-to-be-considered gameplays were limited by knowing which moves are optimal (or which moves cannot be optimal) in certain situations. We therefore leave the reader with another open question: Does the Greedy-move conjecture hold?
\bigskip

\subsection*{Acknowledgments}
The authors would like to express gratitude to Pavel Valtr, Jan Kynčl, and Jan Kratochvíl for helpful discussions. A special thanks goes to Jan Kynčl for numerous comments on this manuscript! Our last but not least thanks goes to Pavel Valtr for the chocolate bet.

Sara Nicholson's work was supported by grant no.~21-32817S of the Czech Science Foundation (GA\v{C}R).
\bigskip


\end{document}